\documentclass[12pt]{amsart}
\usepackage{amssymb}
\usepackage{enumerate,titletoc,mathrsfs}
\usepackage[pagebackref,colorlinks,linkcolor=red,citecolor=blue,urlcolor=blue,hypertexnames=true]{hyperref}
\renewcommand{\setminus}{\smallsetminus}
\renewcommand{\subset}{\subseteq}

\def\CC{\mathbb C}
\def\DD{\mathbb D}
\def\TT{\mathbb T}
\def\fA{\mathfrak{A}}
\def\dtn{\frac{dt}{2\pi}}
\def\BB{\mathbb{B}^2}
\def\PP{\BB}
\def\ww{\rho}
\def\tww{\tilde{\ww}}

\def\intpp{\int_{0}^{2\pi}}

\def\Cp{C_{\rho}}
\def\uu{\gamma}

\def\dd{\sigma}

\newtheorem{theorem}{Theorem}[section]

\newtheorem{lemma}[theorem]{Lemma}

\theoremstyle{definition}

\theoremstyle{remark}

\newtheorem{remark}[theorem]{Remark}

\numberwithin{equation}{section}

         {\begin{exa}}
             {{\hfill $\Box ~~$}\end{example}}

\newcounter{Inc}

\usepackage[dvipsnames]{xcolor}

\newcommand{\df}[1]{{\it{#1}}{\index{#1}}}

\title[The Neil Algebra]{Szeg\"o and Widom Theorems for the Neil Algebra}

\author[S. Balasubramanian]{Sriram Balasubramanian${}^1$}
\address{Sriram Balasubramanian, Department of Mathematics \\ IIT Madras \\ Chennai - 600036, India}
\email{bsriram@iitm.ac.in}
\thanks{${}^1$Supported by the New Faculty Initiative Grant (MAT/15-16/836/NFIG/SRIM) of IIT Madras. ${}^2$Research supported by the NSF grant DMS-1361501}

\author[S. McCullough]{Scott McCullough${}^2$}
\address{Scott McCullough, Department of Mathematics\\
 University of Florida\\ Gainesville %\\
  }
  \email{sam@ufl.edu}
\author[U. Wijesooriya]{Udeni Wijesooriya}
\address{Udeni Wijesooriya, Department of Mathematics\\
 University of Florida\\ Gainesville %\\
  }
  \email{wudeni.pera06@ufl.edu}

\dedicatory{In appreciation for his profound influence on operator theory and our 
mathematical lives, we dedicate this article to Joe Ball.}

\begin{document}

%\makeindex

%%%%%%%%%%

%setcounter{tocdepth}{3}
%contentsmargin{2.55em}
%dottedcontents{section}[3.8em]{}{2.3em}{.4pc}
%dottedcontents{subsection}[6.1em]{}{3.2em}{.4pc}
%dottedcontents{subsubsection}[8.4em]{}{4.1em}{.4pc}

%setcounter{page}{1}

%\begin{document}

\begin{abstract}
Versions of well known function theoretic operator theory results of Szeg\"o and Widom are established for the Neil algebra. The Neil algebra is the subalgebra
of the algebra of bounded analytic functions on the unit disc consisting of those functions whose derivative vanishes at the origin.
\end{abstract}

%\subjclass[2010]{47B335, 30H10 (Primary) 30H05, 46E20 (secondary)}
%\date{\today}

\maketitle

\section{Introduction}
\index{$\CC$} \index{$\DD$} \index{$\TT$} \index{$H^2$} \index{$H^\infty$} \index{$L^p$} \index{$\mathscr{P}$}
Let $\CC$ denote the complex numbers,  $\DD=\{|z|<1\}\subset \CC$ denote the unit disk with its boundary $\TT =\{|z|=1\}$.
Denote by  $H^2=H^2(\DD)$ and $H^\infty=H^\infty(\DD)$  the standard Hardy
spaces of functions analytic in $\mathbb D$ with square summable power series coefficients and bounded analytic functions on $\DD$ respectively.
Let $L^p$ denote the $L^p$ spaces for the $\TT$ (identified with the corresponding $L^p$ spaces for $[0,2 \pi]$ with respect to the measure $\dtn$). Let $\mathscr{P}$ denote the set of analytic polynomials that vanish at $0$.
Thus a $p\in \mathscr{P}$ has the form,
\[
  p(z) =\sum_{j=1}^n p_j z^j
\]
for some positive integer $n$ and $p_1,\dots,p_n\in \CC$.  Given a non-negative function $\ww$ on $\TT$ with $\log(\ww)\in L^1$
a (special case of a) well known result of Szeg\"o (see for instance \cite{Koosis} page 219) identifies the $L^2(\ww)$ distance from the constant function $1$ to $\mathscr{P}.$

\begin{theorem}[of Szeg\"o]
\label{thm:therealszego}
 \[\inf \{ \intpp |p-1|^2 \, \ww \, \dtn: p\in \mathscr{P}\} =\exp(\intpp \log(\ww)\, \dtn).\]
\end{theorem}

A theorem of Widom  characterizes those unimodular functions $\phi \in L^\infty$ whose distance to $H^\infty$ is less than one in terms of Toeplitz operators. A $\phi\in L^\infty$ induces a multiplication operator $M_\phi:L^2\to L^2$ defined by $M_\phi f =\phi f.$
Let $V:H^2\to L^2$ denote the inclusion. The operator $T_\phi = V^* M_\phi V$ is the \df{Toeplitz operator} with symbol $\phi$.
\index{$M_\phi$} \index{$T_\phi$}

\begin{theorem}[Widom's invertibility criteria {\cite[Theorem 7.30]{douglas}}]
\label{thm:realwidom}
  Suppose $\phi \in L^\infty$ is unimodular. There exists an $f\in H^\infty$ such that $\|f-\phi\|<1$ if and only if $T_\phi$ is left invertible.
\end{theorem}

Sarason \cite{sarason} established a version of Theorem \ref{thm:therealszego} for the annulus and Abrahamse \cite[Theorems 4.1 and 4.6]{Ab} established
a version of Theorem \ref{thm:realwidom} for multiply connected domains.  In this paper we establish  Szeg\"o and Widom type theorems for the \df{Neil algebra}. The Neil algebra $\fA$ \index{$\fA$} is the subalgebra of $H^\infty(\mathbb D)$ consisting
of those functions whose derivative vanishes at $0$. It is perhaps the simplest example of a constrained algebra.
As with extending classical results from the unit disc to multiply connected domains,  here it is necessary to replace $H^2$ with a family of Hilbert-Hardy spaces that parameterize the distinction between harmonic functions and the real parts of analytic functions in $\fA$ either explicitly or implicitly in the statement of the results and their proofs. In addition to the references already cited, see for instance \cite{Ab2,Ab3,mrinal,adam} for related results on multiply connected domains,
\cite{BBtH,BH1,BH2,DP, DPRS,mrinal, mrinal2} for results on constrained algebras, \cite{AS}
for results in the context of uniform algebras and finally \cite{JKM} for a Pick interpolation theorem on distinguished varieties.
Let $\fA_0$ denote those functions in $\fA$ that vanish at $0$. Hence $\fA_0=z^2 H^\infty$. \index{$\fA_0$} \index{$\Cp$}

\begin{theorem}[Szeg\"o Theorem for $\fA$]\footnote{\cite[Theorem 5.1]{AS} covers the case $\lambda=0.$}
\label{thm:SzegofA}
 Suppose $\ww>0$ is a continuous function on $\TT$ and let
\[
 \Cp =\intpp \log(\ww)\, \dtn, \ \ \
 \lambda =  \int_{0}^{2\pi} \ww(t)\exp(-it)\, \frac{dt}{2\pi}.
\]
With these notations,
\[
\inf\{ \int_{0}^{2\pi} |1-p|^2 \, \ww \frac{dt}{2\pi}: p\in \fA_0\}  =
\exp(\Cp) + \exp(-\Cp)\, |\lambda|^2.
\]
\end{theorem}

\begin{remark}\rm
Note that $\lambda=0$ if and only if $1$ and $e^{it}$ are orthogonal in $L^2(\rho)$ and in this case it is evident that the distance from $1$ to $\mathscr{P}$
is the same as the distance from $1$ to the subspace $\fA_0$ of $\mathscr{P}$. \qed
%An alternate interpretation of Theorem \ref{thm:SzegofA} is as a formula for the distance from $1$ to $z^2 H^2$ in $L^2(\ww)$.
%\qed
\end{remark}

\index{$\BB$} \index{$\PP$} \index{$H^2_\alpha$} \index{$T^\alpha_\phi$}
To state the analog of Theorem \ref{thm:realwidom} for $\fA$ some notations are needed.
Let $\BB=\{(z,w)\in \CC^2 : |z^2|+|w|^2=1\}$ denote the unit ball in $\CC^2.$
To $\alpha=(a,b)\in \BB$ associate the subspace $H^2_\alpha \subset H^2$ consisting of those $f\in H^2$ such that
\[
 f(0)\, b= f^\prime(0)\, a.
\]
Let $V_\alpha:H^2_\alpha\to L^2$ denote the inclusion. Hence $P_\alpha= V_\alpha V_\alpha^*:L^2\to H^2_\alpha$ is the projection onto $H^2_\alpha$. \index{$P_\alpha$}
Given $\phi\in L^\infty$,
define $T_\phi^\alpha:H^2_\alpha\to H^2_\alpha$ by
\[
 T_\phi^\alpha = V_\alpha^* M_\phi V_\alpha.
\]
It is the \df{Toeplitz operator with symbol $\phi$ with respect to $\alpha$} \cite{adam}.
In particular, if $\phi \in \fA$ and $f\in H^2_\alpha$,  then $V^* T^\alpha_\phi f = \phi f=T^\alpha_\phi f.$

\begin{remark}\rm
\label{rem:PP}
%There is of course redundancy built into the parameterization of the spaces $H^2_\alpha$ by $\BB$.
Given $\alpha =(a,b)$ and $\beta=(c,d)$, if $ad=bc$, then $H^2_\alpha =H^2_\beta$ and likewise
$T^\alpha_\phi=T^\beta_\phi$. Thus, $\mathbb P$, complex projective space obtained by moding out $\BB$ by the relation $(a,b)=(c,d)$,
is a natural choice of parameter space. For ease of exposition we accept the redundancy inherent in the use of $\BB$. \qed
\end{remark}

\begin{theorem}[Inversion for $\fA$]
\label{thm:widomfA}
 Suppose $\phi\in L^\infty$ is unimodular.
The distance from $\phi$ to $\fA$ is strictly less than one if and only if
$T^\alpha_\phi$ is left invertible for each $\alpha\in\BB.$ Likewise, the distance from $\phi$ to the invertible elements of $\fA$ is strictly less than one if and only
if $T^\alpha_\phi$ is invertible for each $\alpha\in \BB$.
\end{theorem}

Before turning to the proofs of Theorems \ref{thm:SzegofA} and \ref{thm:widomfA}, we pause to introduce some conventions
and basic background on the spaces $H^2_\alpha.$
For $p=2,\infty$,
the standard identification of $H^p(\DD)$ with $H^p(\TT)$, where the latter is viewed as the subspace of $L^p(\TT)$ consisting
of those $f$ with vanishing negative Fourier coefficients, will be used routinely and without comment.
%  Ditto for the identification of $H^2(\DD)$ with $H^2(\TT)$ where the latter is the subspace of $L^\infty(\TT)$ consisting
%of those functions with vanishing negative Fourier coefficients.
Let $H^2_1$ denote the subspace of $H^2$ consisting of those  $f\in H^2$ whose Fourier coefficient
\[
 \hat{f}(1) = \intpp f\, e^{-it} \, \dtn = 0.
\]
Evidently, $H^2_1$ is the closure of $\fA$ in $H^2$.
\index{$H^2(\TT)$} \index{$H^2_1$} \index{$\hat{f}(1)$}
The following Lemma can be found in \cite{DPRS} for instance.
The first part follows from the easily
verified fact that $\{a+bz,z^n:n\ge 2\}$ is an orthonormal basis for $H^2_\alpha$;
and the moreover part, from a standard reproducing kernel Hilbert space argument.

\begin{lemma}
 \label{lem:ker0}
  For each $\alpha=(a,b)\in \PP$,  the space $H^2_\alpha$ has reproducing kernel,
\[
 k^\alpha_w(z) = k^\alpha(z,w) = (a+bz)\overline{(a+bw)} + \frac{z^2 {\overline{w^2}}}{1-z\overline{w}}, \ \ \ z,w\in \DD.
\]
In particular,
\[
\|k^\alpha_0\|^2 = k^{\alpha}(0,0)=|a|^2,
\]
and thus  $k^\alpha_w\ne 0$ with the exception of $\alpha=(0,1)$ and $w=0.$

Moreover, if $\psi\in \fA$ and $w\in\DD$, then $(T^\alpha_\psi)^* k^\alpha_w = \overline{\psi(w)}k^\alpha_w$.
\end{lemma}

\section{Proof of Theorem \ref{thm:SzegofA}}
As a first step, observe that it suffices to prove the theorem under the additional hypothesis that $\Cp=0$.
Indeed, if not, let
 $\tww=\exp(-\Cp)\, \rho$, so that  $\intpp \log(\tww) \,\dtn =0.$  In particular, $C_{\tww}=0$ and  with
\[
 \tilde{\lambda} = \intpp \tww \, \exp(-it)\, \dtn = \exp(-\Cp)\,  \lambda,
\]
if Theorem \ref{thm:SzegofA} holds for $\tww$, then
\[
 \inf \{ \intpp |p-1|^2 \, \tww \, \dtn: p\in \fA_0\} = 1+|\tilde{\lambda}|^2.
\]
Thus,
\[
\begin{split}
 \inf \{ \intpp |p-1|^2 \, \ww \, \dtn: p\in \fA_0\}
 = &\exp(\Cp) \inf \{ \intpp |p-1|^2 \, \tww \, dtn: p\in \fA_0\}\\
 = & \exp(\Cp) (1+|\tilde{\lambda}|^2) = \exp(\Cp)+\exp(-\Cp)\, |\lambda|^2
\end{split}
\]
as claimed.   Accordingly, for the remainder of the proof,  assume $\Cp=0$.

Let
\[
\dd = \frac{1}{\sqrt{1+|\lambda|^2}} \, (1,\lambda) \in \BB.
\]
In particular,
\[
 \|k^{\dd}_0\|^2 = \frac{1}{1+|\lambda|^2}.
\]

Note that, as sets, $L^2(\ww)$ and $L^2$ are the same and thus we may consider $H^2$ as a Hilbert space
with the alternate inner product,
\[
 \langle f,g\rangle_\ww = \intpp f\overline{g}\, \ww\dtn.
\]
To keep the distinction clear, denote this latter space by $H^2(\ww)$.
Since the closure of $\fA_0$ in $H^2(\ww)$ is $z^2 H^2=z^2 H^2(\ww)$, the objective
is to find the $H^2(\rho)$-distance from $1$ to $z^2 H^2.$ That is, to show
\[
\inf \{ \intpp |p-1|^2 \, \tww \, \dtn: f\in z^2 H^2\} = 1+|\tilde{\lambda}|^2.
\]

Since $\ww$ is continuous and strictly positive, $\log(\ww)$ is continuous. It has Fourier series expansion
\[
 \log(\ww) = \sum_{j=-\infty}^\infty c_j  e^{ijt},
\]
where, because it is real-valued, $c_{-j}=\overline{c_j}.$ % and of course $\sum |c_j|^2<\infty$.
Moreover, $c_0=0$
and $c_1=\lambda,$ since $\Cp=0$ and by the very definitions of $\Cp$ and $\lambda.$ Letting $\uu$ denote the $H^2$ function represented
by the series
\[
 \uu = \sum_{j=1}^\infty c_j e^{ijt},
\]
it follows that $\log(\ww) = \uu +\uu^*$ as elements of $L^2$.
Further, since
\[
 |\exp(\pm \uu)|^2 = \exp\left(\pm (\uu+\uu^*)\right) = \ww^{\pm 1},
\]
both $\exp(\pm \uu)$ are in $H^\infty$.
The mapping $U:H^2(\rho)\to H^2$ defined by $Uf=\exp(\uu) f$ is a unitary map
with inverse $U^*f=\exp(-\uu) f$. Moreover, $U(z^2H^2)=z^2 H^2$.
Thus, the aim is to find the $H^2$-distance from $\exp(\uu)$ to $z^2H^2$.

Given $f\in z^2 H^2$, let $g=\exp(\uu)-f$ and estimate, using $g(0)=1$ and the Cauchy-Schwarz inequality,
\begin{equation}
\label{eq:usesCS-alt}
 \begin{split}
  \|\exp(\uu)-f\|^2
     =  & \|g\|^2 \\
      \ge & \frac{ |\langle g, k^{\dd}_0 \rangle|^2}{ \|k^{\dd}_0\|^2} \\
   = & |g(0)|^2\, (1+|\lambda|^2) \\
   = & 1+|\lambda|^2.
 \end{split}
\end{equation}
Let
\[
 f= \exp(\uu)- (1+|\lambda|^2)k_0^{\dd}
\]
and note $f(0)=0$ and $f^\prime(0)=\uu^\prime(0)-\lambda=0$. Thus $f\in z^2 H^2$ and, with this choice of
$f$, equality holds in the Cauchy-Schwarz inequality in equation \eqref{eq:usesCS-alt}.

\section{Toeplitz operators on $\fA$}
This section contains the proof of Theorem \ref{thm:widomfA}.

\begin{lemma}
\label{lem:norm}
If $\phi\in L^\infty$,  then  $\|T^\alpha_\phi\| =\|\phi\|$ and
$(T^\alpha_\phi)^* = T^\alpha_{\overline{\phi}}$.
\end{lemma}

\begin{proof}
Since $M_{\phi}^* = M_{\overline{\phi}}$, it follows that $(T^\alpha_\phi)^* = V_{\alpha}^* M_{\phi}^* V_{\alpha} = V_{\alpha}^* M_{\overline{\phi}} V_{\alpha} =  T^\alpha_{\overline{\phi}}.$
Since $V_{\alpha}$ is an isometry, it follows that $\|T^\alpha_\phi\| \le \|M_{\phi}\| = \|\phi\|.$  Now let $V:H^2\to L^2$ and $W:z^2 H^2\to L^2$ denote the inclusion maps.
In particular, $V^*M_\phi V$ is $T_\phi,$ the usual Toeplitz operator with symbol $\phi$.  On the other hand,
$W^*M_\phi W = W^* T^\alpha_\phi W$.  With $U:H^2\to z^2 H^2$ given by $Uf=z^2f$, it follows that $U$ is unitary and, for $f,g\in H^2$,
\[
 \langle M_\phi WUf,WUg\rangle = \langle z^2\phi f,z^2 g\rangle = \langle \phi f,g\rangle = \langle  M_\phi f,g\rangle =\langle V^* M_\phi Vf,g\rangle.
\]
Hence $U^* W^* M_\phi W U = V^* M_\phi V=T_\phi$ and consequently $W^* T^\alpha_\phi W$ is unitarily equivalent to $T_\phi$.  Hence $\|T^\alpha_\phi\|\ge \|T_\phi\|$. Since,  as is well known  that $\|T_\phi\|=\|\phi\|$ (\cite{M-AR}), the result follows.
\end{proof}

Let $\mathscr{B}(L^2)$ denote the bounded linear operators on $L^2.$ % and recall $P_\alpha$ is the projection from $L^2$ onto $H^2_\alpha$.

\begin{lemma}
\label{lem:Pcont}
Giving $\BB$ its usual topology and $\mathscr{B}(L^2)$ its norm topology, the mapping $\BB \ni \alpha \to P_\alpha\in \mathscr{B}(L^2)$ is continuous.
\end{lemma}

\begin{proof}
Since $\{a+bz,z^n:n\ge 2\}$ is an orthonormal basis for $H^2_\alpha$, if
 $f=\sum f_nz^n \in H^2$ and $\alpha =(a,b)\in \BB$, then
\[
 P_\alpha f = (\overline{a}f_0 + \overline{b}f_1)(a+bz) + \sum_{n=2}^\infty f_n z^n.
\]
Thus, letting $Q$ denote the projection onto $z^2 H^2$ and $F_\alpha = (a+bz)$ (a unit vector),
\[
 P_\alpha = F_\alpha F_\alpha^* + Q,
\]
 where $F_\alpha F_\alpha^*:L^2\to L^2$ is the rank one projection operator,
\[
 F_\alpha F_\alpha^* f = \langle f,F_\alpha\rangle F_\alpha = (\overline{a} f_0 +\overline{b} f_1) F_\alpha.
\]
Thus, if $\beta=(c,d)\in \BB$, then
\[
 P_\alpha -P_\beta = F_\alpha F_\alpha^* - F_\beta F_\beta^* = F_\alpha(F_\alpha-F_\beta)^* + (F_\alpha-F_\beta)F_\beta^*.
\]
Since $\|F_\alpha -F_\beta\| = \|\alpha -\beta\|$, the result follows.
\end{proof}

%Let $\overline{z}H^2_\alpha$ denote the subspace of $L^2(\mathbb T)$ consisting of functions $f$ with Fourier transforms of the form,
Let $\mathscr{M}\subset L^1$ denote the subspace consisting of 
those $L^1$ functions with Fourier series of the form \index{$\mathscr{M}$}
\begin{equation}
\label{eq:dual}
  \hat{f}(-1)\exp(-it) +\sum_{j=1}^\infty \hat{f}(j) \exp(ijt).
\end{equation}
The following lemma is the $\mathscr{M}$ version of the well known factorization theorem for $H^1$ functions.

\begin{lemma}
\label{lem:factor}
 If $h\in \mathscr{M}$, then there exist
\begin{enumerate}[(i)]
\item  $\alpha\in\PP;$
\item  $f\in H^2_\alpha;$ and
\item  $g\in L^2$
\end{enumerate}
such that
\begin{enumerate}[(a)]
 \item  $\overline{g} \in (H^2_\alpha)^\perp$;
 \item $h=fg$;  and
 \item $\|h\|_1 =\|f\|_2\, \|g\|_2$.
\end{enumerate}
\end{lemma}

\begin{proof}
 The function $\psi=zh$ is in $H^1$ and therefore there exists $F,G\in H^2$ such that $zh=FG$ and $\|h\|_1 = \|\psi\|_1 = \|F\|_2\, \|G\|_2$
\cite[Corollary 6.27]{douglas}.
 Moreover, since $\psi^\prime(0)=0$, it follows that $F^\prime(0)\, G(0)+F(0)\, G^\prime(0)=0$. There is an $\alpha=(a,b)\in \BB$ such that $F\in H^2_\alpha$.
(Indeed, simply choose $\alpha\in \PP$ such that $a F^\prime(0)=bF(0)$.)  Thus there is a constant $c$ and an $H^2$ function $F_0$ such that
\[
 F=c(a+bz) + z^2 F_0.
\]
Hence, there is a constant $d$ and $H^2$ function $G_0$ such that
\[
G=d(a-bz)+z^2 G_0.
\]
Let $g=\overline{z} G,$ in which case $h=Fg$ and $\|g\|_2 =\|G\|_2$. Moreover,
\[
\langle a+bz,\overline{g}\rangle = d \intpp (a+bz)\left ( d(a\overline{z}-b) + zG_0\right)\, \dtn =0
\]
and, for $n\ge 2$,
\[
\langle z^n, \overline{g}\rangle =\intpp z^n \left(d(a\overline{z}-b)+zG_0\right)\, \dtn =0.
\]
Hence $\overline{g}\in (H^2_\alpha)^\perp$.
\end{proof}

Recall   $(L^1)^*=L^\infty$  with the equality interpreted as the isometric isomorphism determined by the mapping that
assigns to $\phi\in L^\infty$  the linear functional $\lambda_\phi:L^1\to\CC$ given by
\[
\lambda_\phi(\psi) = \intpp \phi \, \psi \, \dtn.
\]
%to a $\phi\in L^\infty$
Moreover, letting
\[
\mathscr{M}^\perp :=\{\phi\in L^\infty: \intpp \phi \, \psi \dtn = 0, \, \mbox{ for all } \psi\in \mathscr{M}\},
\]
and $\pi:L^\infty \to L^\infty/\mathscr{M}^\perp$ denote the quotient mapping,
the mapping $\Lambda:L^\infty/\mathscr{M}^\perp \to \mathscr{M}^*$ given by
\[
 \Lambda(\pi(\lambda_\phi)) = (\lambda_\phi)|_{\mathscr{M}},
\]
is an isometric isomorphism. Finally, if $\phi\in \mathscr{M}$ and $\psi\in \fA$, then
\[
 \intpp \phi\, \psi\, \dtn =0.
\]
Thus, $\fA\subset \mathscr{M}^\perp.$ On the other hand, $e^{ijt}\in \mathscr{M}$ for $j=-1,1,2,\dots$
and therefore if $\psi\in \mathscr{M}^\perp$, then its Fourier series has the form
\[
 \psi = \hat{\psi}(0) + \sum_{j=2}^\infty \hat{\psi}(j) e^{ijt}.
\]
Hence $\psi\in \fA$  and thus we may view $\Lambda$ as having domain $L^\infty/\fA$. The following lemma
summarizes the discussion (see \cite[page 88]{conway}).

\begin{lemma}
\label{lem:dual}
 $\Lambda:L^\infty /\fA\to \mathscr{M}^*$ defined by sending $\pi(\phi)$ to the linear functional $\tilde{\lambda}_\phi:\mathscr{M}\to \CC$ given by
\[
 \tilde{\lambda}_\phi(f) = \intpp \phi \, f \, \dtn
\]
 is an isometric isomorphism.
\end{lemma}

%  \begin{proof}
% Fix $0\ne \phi \in \mathscr{M}$.
% Let $\psi \in \fA$ be given. Since $|\lambda_\phi(\pi(f))| = \vert \int f \phi \dtn \vert = \vert \int (f + \psi) \phi \dtn \vert \le \|f+\psi\| \|\phi\|$. Since $\psi \in \fA$ is arbitrary, it follows that $|\lambda_\phi(\pi(f))| \le \inf_{\psi \in \fA} \|f+\psi\| \,  \|\phi\| = \|(\pi(f)\| \|\phi\|$. Thus $\|\lambda_\phi\| \le \|\phi\|$. Let $f=\overline{\psi} |\psi|^{-1}$ if $\psi\ne 0$ and $0$ if $\psi=0$. It follows that $\|f\|_\infty = 1$ and at the same time
% \[
%  \|\lambda_\phi(\pi(f))\| = \|\phi\|_1.
% \]
% Hence, $\|\lambda_\phi\|\ge \|\phi\|_1$ and thus $\Lambda$ is an isometric mapping.

% Now suppose $\lambda\in  (L^\infty/\fA)^*$.  From the general theory, the mapping $\sigma: (L^\infty/\fA)^* \to \fA^*$ given by $\sigma(\lambda)=\lambda\circ \pi$
% is an isometric isomorphism.

% % \textcolor{red}{Should we assume $\phi$ to be non-vanishing?} If so, then $\|\lambda_\phi\left(\pi\left(\frac{\overline{\phi}}{|\phi|}\right)\right)\| = \|\phi\|$, the proof is complete.
% \end{proof}

\begin{lemma}
\label{lem:product}
 If $\phi\in L^\infty$ and $\psi\in \fA$, then
\[
\begin{split}
 T^\alpha_{\overline{\psi}\phi}=& T^\alpha_{\overline{\psi}}\, T^\alpha_\phi\\
 T^\alpha_{\psi\overline{\phi}} = & T^\alpha_{\overline{\phi}} \,  T^\alpha_\psi.
\end{split}
\]
\end{lemma}

\begin{proof}
Let $f, g \in H^2_\alpha$ be given. Since $\psi g \in H^2_{\alpha}$, it follows, using Lemma \ref{lem:norm},  that $\langle T^\alpha_{\overline{\psi}}\, T^\alpha_\phi f, g \rangle =
\langle T^\alpha_\phi f, (T^\alpha_{\overline{\psi}})^*g \rangle = \langle T^\alpha_\phi f, T^\alpha_{{\psi}}g \rangle
%= \langle V_{\alpha}^* M_{\phi}V_{\alpha}f, V_{\alpha}^* M_{\psi}V_{\alpha}g \rangle
= \langle V_{\alpha}^* \phi f, \psi g \rangle = \langle \phi f, V_{\alpha} \psi g \rangle =
\langle \phi f, \psi g \rangle = \langle \overline{\psi} \phi f, g \rangle = \langle \overline{\psi} \phi V_{\alpha}f, V_{\alpha} g \rangle = \langle T_{\overline{\psi}\phi}^{\alpha} f, g \rangle$. Thus $T^\alpha_{\overline{\psi}\phi}= T^\alpha_{\overline{\psi}}\, T^\alpha_\phi$. Applying Lemma \ref{lem:norm} to what has already been proved, $T^\alpha_{\psi\overline{\phi}} = (T^\alpha_{\overline{\psi}\phi})^* = (T^\alpha_{\overline{\psi}}\, T^\alpha_\phi)^* = T^\alpha_{\overline{\phi}} T^\alpha_{\psi}$.
\end{proof}

An element $\psi\in \fA$ is \df{invertible in $\fA$} if it does not  vanish in $\DD$ and $\psi^{-1}=\frac{1}{\psi}\in \fA$.

\begin{lemma}
\label{lem:inverse}
Suppose $\psi\in \fA$. The following are equivalent.
\begin{enumerate}[(i)]
 \item \label{it:invertible} $\psi$ is invertible in $\fA$;
 \item \label{it:somea} there is an $\alpha\in \PP$ such that $T^\alpha_\psi$ is right invertible;
 \item \label{it:alla}$T^\alpha_\psi$ is invertible for each $\alpha \in \PP$.
\end{enumerate}
Moreover, in this case  $(T^\alpha_\psi)^{-1} = T^\alpha_{\psi^{-1}}$.
\end{lemma}

\begin{proof}
 Evidently item \eqref{it:invertible} implies item \eqref{it:alla} implies item \eqref{it:somea}.
 Now suppose there is an $\alpha$ such that $T:=T^\alpha_\psi$ is right invertible.
 The Hilbert space $H^2_\alpha$ has a reproducing kernel $k^\alpha_w(z)$ and further
$T^* k^\alpha_w = \overline{\psi(w)}k^\alpha_w$ by Lemma \ref{lem:ker0}. Since $T$ is right invertible,  $T^*$
is bounded below; i.e., there is a $\delta>0$  such that $\|T^*f\|\ge \delta\|f\|$ for all $f\in H^2_\alpha$. Hence,
\[
 |\overline{\psi(w)}| \, \|k^\alpha_w\|
    =\|T^* k^\alpha_w\|\ge \delta \|k^\alpha_w\|.
\]
 Moreover, by Lemma \ref{lem:ker0} $k^\alpha_w\ne 0$ for $w\ne 0.$  Thus $|\frac{1}{\psi}(w)|\le \frac{1}{\delta}$ for
$w\in \DD\setminus \{0\}$ and therefore, as $\frac{1}{\psi}$ is otherwise analytic,  $|\frac{1}{\psi}|$ is bounded by $\frac{1}{\delta}$.
Since $\psi\in \fA$ it follows that $\frac{1}{\psi}\in \fA$ too; i.e., item \eqref{it:invertible} holds.
\end{proof}

\begin{lemma}
\label{lem:leftinvertible}
 Suppose $\phi\in L^\infty$ is unimodular. If there exists $\psi\in \fA$ such that $\|\phi-\psi\|<1$, then $T^\alpha_{\overline{\phi}}T^\alpha_\psi$  is invertible,
and therefore $T^\alpha_\phi$ is left invertible, for each $\alpha\in \PP$.
 Further, if $\psi$ is invertible in $\fA$, then $T^\alpha_{\overline{\psi}}\,T^\alpha_\phi$ is invertible, and therefore $T^\alpha_{\phi}$ is invertible, for each $\alpha\in \PP$.
\end{lemma}

\begin{proof}
%In view of Lemma \ref{lem:norm}, it suffices to show $T^\alpha_{\overline{\phi}}$ is right invertible and invertible respectively.
Suppose there exists $\psi\in \fA$ such that $\|\phi-\psi\|<1$. In this case $\|1-\psi \overline{\phi}\|<1$, since $|\phi|=1$ (unimodular).
Hence, by Lemma \ref{lem:norm}, for a given $\alpha\in \PP$,
\[
 1 > \|1-\psi \overline{\phi}\| = \|T^\alpha_{1-\psi\overline{\phi}}\|=\|1-T^\alpha_{\psi \overline{\phi}}\|.
\]
In particular, $T^\alpha_{\psi \overline{\phi}}$ is invertible.
Since $\psi\in \fA$, Lemma \ref{lem:product} applies to give, $T^\alpha_{\psi\overline{\phi}}=T^\alpha_{\overline{\phi}}T^\alpha_\psi$. Thus  $T^\alpha_{\overline{\phi}}$
is right invertible. By Lemma \ref{lem:norm},  $(T^\alpha_{\overline{\phi}})^* = T^\alpha_\phi$  is left invertible.
%and thus $T^\alphais right invertible and hence, using Lemma \ref{lem:norm}, $T^\alpha_\phi$ is left invertible.

 Now, assuming $\psi$ is invertible in $\fA$,
by Lemma \ref{lem:inverse}, $T^\alpha_\psi$ is invertible.  The invertibility of  $T^\alpha_{\overline{\phi}}$ follows.
Thus, again using Lemma \ref{lem:norm}, $T^\alpha_\phi$ is invertible.
\end{proof}

\begin{lemma}
\label{lem:boundedbelow}
 If $\phi\in L^\infty$ and  $T^\alpha_\phi$ is left invertible for each $\alpha\in\PP$, then there exists an $\epsilon \in (0, 1]$, such that
for each $\alpha\in \PP$ and $f\in H^2_\alpha$,
\[
\|T^\alpha_\phi f\|\ge \epsilon \|f\|.
\]
\end{lemma}

\begin{proof}
 For $\alpha \in \BB$, define $X_\alpha:L^2\to L^2$ by $X_\alpha = P_\alpha M_\phi P_\alpha + (I-P_\alpha)$.
 Given $\alpha\in\BB$,  since $T^\alpha_\phi$ is left invertible, there exists an $\epsilon_\alpha \in (0, 1]$ such that
 $\|{V_{\alpha}}T^\alpha_\phi f\|=\|T^\alpha_\phi f\|\ge \epsilon_\alpha \|f\|$ for $f\in H^2_\alpha$. Hence, given $F=f+g$ with $f\in H^2_\alpha$ and $g\in (H^2_\alpha)^\perp$,
\[
 \|X_\alpha F\|^2 = \|{V_{\alpha}}T^\alpha_\phi f\|^2 +\|g\|^2 \ge \epsilon_\alpha^2 \|F\|^2.
\]
Thus,  $\|X_\alpha F\|\ge \epsilon_\alpha \|F\|$ for all $F\in L^2$.

To show there is an $\epsilon>0$ such that $\|X_\alpha F\|\ge \epsilon \|F\|$ for all $\alpha \in \PP$ and $F\in L^2$, we argue by contradiction.
Accordingly suppose no such $\epsilon>0$ exists.
 By compactness of $\BB$,  there is a sequence $\alpha_n=(a_n,b_n)$ from $\BB$,  that, by passing to a subsequence if needed,
we may assume converges to some $\beta=(a,b)\in \BB$ and a unit vectors $F_n \in L^2$ such that $(\|X_{\alpha_n}F_n\|)_n$ converges to $0$.
But then,
\[
0< \epsilon_\beta \le  \|X_\beta F_n\|\le \|X_{\alpha_n}F_n\| + \|(X_{\beta}-X_{\alpha_n})F_n\|.
\]
By norm continuity (Lemma \ref{lem:Pcont}) the last term on the right hand side tends to $0$ and by assumption the first term on the right hand side
tends to $0$, a contradiction.

To complete the proof, simply observe if $f\in H^2_\alpha\subset L^2$, then $\|\phi\| \|f\| \ge \|T^\alpha_\phi f\|=\|X_\alpha f\|\ge \epsilon \|f\|$.
\end{proof}

\begin{lemma}
\label{lem:left}
 Suppose $\phi\in L^\infty$ is unimodular. The distance from $\phi$ to  $\fA$ is strictly less than one if and only
if $T^\alpha_\phi$ is left invertible for each $\alpha\in \PP$.
\end{lemma}

\begin{proof}
Suppose $T^\alpha_\phi$ is left invertible for each $\alpha\in\PP$. In this case, Lemma \ref{lem:boundedbelow} applies and thus
there is an $1 \ge \epsilon>0$ such that for each $\alpha$ and $f\in H^2_\alpha$,
\[
 \|T^\alpha_\phi f\|\ge \epsilon \|f\|.
\]
%(Moreover, by choosing an $f \in H^2_{\alpha}$ of norm one, it follows, using Lemma \ref{lem:norm}, that $1 = \|\phi\|  = \|T^\alpha_\phi \| \ge \|T^\alpha_\phi f\|\ge \epsilon;$ i.e., $0<\epsilon \le 1$.)

Now let $h\in\mathscr{M}$ be given. By Lemma \ref{lem:factor} there is an $\alpha\in \PP$ and $f\in H^2_\alpha$ and a $g\in L^2$ such that $\overline{g}\in (H^2_\alpha)^\perp$
and both $h=fg$ and $\|h\|_1 = \|f\|_2\,\|g\|_2$.  Thus,
\[
\begin{split}
\left | \intpp \phi h \,\dtn \right |
  = & \left | \intpp \phi fg \, \dtn \right |\\
 = &\left | \langle \phi f,\overline{g}\rangle \right | \\
 = & \left | \langle \phi f,(I-P_\alpha)\overline{g}\rangle \right | \\
 = & \left | \langle (I-P_\alpha)\phi f,\overline{g} \rangle \right |\\
\le & \|(I-P_\alpha)\phi f\| \, \|g\|.
\end{split}
\]
On the other hand, using the unimodular hypothesis,
\[
\begin{split}
 \|f\|^2= & \,  \|\phi \, f\|^2= \|P_\alpha \phi f\|^2 +\|(I-P_\alpha)\phi\, f\|^2 \\
 = & \, \|T^\phi_\alpha f\|^2 + \|(I-P_\alpha)\phi \, f\|^2 \\
 \ge & \, \epsilon^2 \|f\|^2  + \|(I-P_\alpha)\phi \, f\|^2.
\end{split}
\]
Thus, $(1-\epsilon^2) \|f\|^2 \ge \|(I-P_\alpha)\phi\, f\|^2.$
Therefore,
\[
 \left | \intpp \phi h \,\dtn \right | \le \sqrt{1-\epsilon^2} \, \|f\|_2\,\|g\|_2 =
  \sqrt{1-\epsilon^2}\, \|h\|_1.
\]
By Lemma \ref{lem:dual}, it now follows that $\|\pi(\phi)\|<1$, where $\pi:L^\infty \to L^\infty/\fA$ is the quotient map; i.e.,
the distance from $\phi$ to $\fA$ is less than one.

{Conversely, if the distance from $\phi$ to $\fA$ is less than one, then there exists a $\psi \in \fA$ such that $\|\phi - \psi\| < 1$. It follows from Lemma  \ref{lem:leftinvertible} that $T^\alpha_\phi$ is left invertible.}
\end{proof}

\begin{proof}[Proof of Theorem \ref{thm:widomfA}]
All that remains to be shown is:  $T^\alpha_\phi$ is invertible for each $\alpha\in\PP$ if and only if the distance from $\phi$ to the invertible
elements of $\fA$ is at most one. If $T^\alpha_\phi$ is invertible for each $\alpha \in \PP$, then there exists a $\psi\in \fA$ such
that $\|\phi-\psi\|<1$ by Lemma \ref{lem:left}.  By Lemma \ref{lem:leftinvertible},  $T^\alpha_{\overline{\phi}} \, T^\alpha_{\psi}$ is invertible.
By Lemma \ref{lem:norm}, $T^\alpha_{\overline{\phi}}$ is invertible  and thus $T^\alpha_\psi$ is invertible.
B Lemma \ref{lem:inverse} $\psi$ is invertible in $\fA$.

The converse is contained in Lemma \ref{lem:leftinvertible}.
\end{proof}
%
%
% In this case, by Lemma \ref{lem:left}, there is a $\psi\in \fA$ such that $\|\phi-\psi\|<1$.  Hence,
%by Lemma \ref{lem:leftinvertible}, $T^\alpha_{\overline{\phi}} \, T^\alpha_{\psi}$ is invertible. On the other hand,

%Now suppose there is a $\psi\in \fA$ that is invertible in $\fA$ such that $\|\phi-\psi\|<1$.

%{The converse follows from the observation that there exists a $\psi \in \fA$ that is invertible such that $\|\phi - \psi \| < 1$, and  an application of Lemma \ref{lem:leftinvertible}.}
%\end{proof}

\printindex
\end{document}